\documentclass[a4paper,reqno,12pt]{amsart} %

\usepackage{amsfonts}
\usepackage{amssymb}
\usepackage{mathrsfs}
\usepackage{amsmath}
\usepackage{dsfont} 
\usepackage{paralist}
\usepackage{microtype}

\usepackage[pagewise]{lineno}

\usepackage{hyperref}
\hypersetup{urlcolor = red, citecolor = blue}
\hypersetup{
  pdfauthor = Xuan Mao, 
  pdftitle = {Finite-time blowup in indirect chemotaxis}, 
  pdfsubject = Manuscript, 
  pdfkeywords = {chemotaxis; indirect signal production; finite-time blowup}
}
\usepackage[top = 1in,bottom = 1in,left = 1in,right = 1in]{geometry}
\usepackage[sort]{cite}

\allowdisplaybreaks[4]

\newtheorem{theorem}{Theorem}[section]

\newtheorem{lemma}[theorem]{Lemma}
\newtheorem{proposition}[theorem]{Proposition}

\theoremstyle{definition}

\newtheorem{remark}[theorem]{Remark}

\numberwithin{equation}{section}

\newcommand{\dd}{\;\mathrm{d}}

\DeclareMathOperator{\loc}{loc}

\begin{document}

\title[Finite-time blowup in indirect chemotaxis]{Finite-time blowup in a fully parabolic chemotaxis model involving indirect signal production
}

\author[X. Mao]{Xuan Mao}
\address{School of Mathematics, Southeast University, Nanjing 211189, P. R. China}
\email{230218181@seu.edu.cn}

\author[M. Liu]{Meng Liu}
\address{Department of Applied Mathematics, Anhui University of Technology, Ma'anshan
243002, P. R. China}
\email{2280563463@qq.com}

\author[Y. Li]{Yuxiang Li}
\address{School of Mathematics, Southeast University, Nanjing 211189, P. R. China}
\email{lieyx@seu.edu.cn}

\thanks{Supported in part by National Natural Science Foundation of China (No. 12271092, No. 11671079) and the Jiangsu Provincial Scientific Research Center of Applied Mathematics (No. BK20233002).}

\subjclass[2020]{35B44; 35K51; 35Q92; 92C17}%
\keywords{chemotaxis; indirect signal production; finite-time blowup}

\begin{abstract}
This paper is concerned with a parabolic-parabolic-parabolic chemotaxis system with indirect signal production, modelling the impact of phenotypic heterogeneity on population aggregation
\begin{equation*}
  \begin{cases} 
u_t = \Delta u - \nabla\cdot(u\nabla v),\\ 
v_t = \Delta v - v + w,\\ 
w_t = \Delta w - w + u,
  \end{cases}
\end{equation*} 
posed on a ball in $\mathbb R^n$ with $n\geq5$, 
subject to homogeneous Neumann boundary conditions.
The system has a four-dimensional critical mass phenomenon concerning blowup in finite or infinite time 
according to the seminal works of Fujie and Senba [J. Differential Equations, 263
(2017), 88--148; 266 (2019), 942--976]. 
We prove that for any prescribed mass $m > 0$, there exist radially symmetric and nonnegative initial data $(u_0,v_0,w_0)\in C^0(\overline{\Omega})\times C^2(\overline{\Omega})\times C^2(\overline{\Omega})$ with $\int_\Omega u_0 = m$ such that the corresponding classical solutions blow up in finite time. The key ingredient is a novel integral inequality for the cross-term integral $\int_\Omega uv$ 
constructed via a Lyapunov functional. 
\end{abstract}

\maketitle

\section{Introduction}\label{introduce section}

This paper is concerned with finite-time singularity formation of classical solutions to the following parabolic-parabolic-parabolic chemotaxis model \eqref{sys: ks isp pp/ep/e} ($\tau=1$ and $\varepsilon=1$) accounting for indirect signal production 
\begin{align}
  \begin{cases}
    \label{sys: ks isp pp/ep/e}
      u_t = \Delta u - \nabla \cdot(u\nabla v),&  x\in\Omega, t>0,\\
      v_t =  \Delta v - v + w,&   x\in\Omega,	t>0,\\
      \tau w_t  = \varepsilon\Delta w - w + u, &   x\in\Omega, t > 0,\\
      \partial_\nu u = \partial_\nu v = \varepsilon\partial_\nu w = 0 , &  x\in\partial\Omega, t >0,\\
      (u(\cdot, 0), v(\cdot, 0), \tau w(\cdot,0)) = (u_0, v_0, \tau w_0), & x\in\Omega,
  \end{cases}
\end{align}
in a bounded domain $\Omega\subset\mathbb R^n$ with smooth boundary for some $n\in\mathbb N$, 
where $\partial_\nu$ denotes the derivative with respect to the outward normal of $\partial\Omega$,
and the initial functions $u_0\in C^0(\overline\Omega)$ and 
$v_0, w_0\in C^2(\overline\Omega)$ are nonnegative with $\partial_\nu v_0 = \partial_\nu w_0 = 0$ on $\partial\Omega$. Here, $\tau,\varepsilon\in\{0,1\}$.

The system~$\eqref{sys: ks isp pp/ep/e}_{\tau,\varepsilon}$ was proposed to describe complex and realistic biological phenomena involving cluster attack of mountain pine beetles~\cite{Strohm2013} $\eqref{sys: ks isp pp/ep/e}_{1,0}$ ($\varepsilon=1$ and $\tau=0$) 
and modeling effects of phenotypical heterogeneity~\cite{Macfarlane2022} \eqref{sys: ks isp pp/ep/e}. 
The model~$\eqref{sys: ks isp pp/ep/e}_{1,0}$ depicts, for example, taxis-driven migration processes of two beetle phenotypes: flying beetles performs chemotactic movement towards location of higher concentration $v$ of pheromones mediated by nesting beetles, 
where $u$ and $w$ denote the density of flying beetles and nesting beetles, respectively. 
See the recent review \cite{Winkler2025} for other biological backgrounds of $\eqref{sys: ks isp pp/ep/e}_{\tau,\varepsilon}$, e.g., a distributed delayed diffusion modelling spatial memory movement of animals~\cite{Shi2021} with dependence on the past time of weak type $\eqref{sys: ks isp pp/ep/e}_{0,0}$ and strong type~\eqref{sys: ks isp pp/ep/e}.

\subsection*{Finite-time blowup \& Lyapunov functional}
The minimal chemotaxis model~$\eqref{sys: ks isp pp/ep/e}_{0,0}$ ($\tau = 0$ and $\varepsilon = 0$) was originally proposed by Keller and Segel~\cite{Keller1970} to investigate the aggregation of cells. Besides random walk, migration of cells $u$ is oriented along the gradient of chemical concentration $v$ secreted by themselves. This self-organized system~$\eqref{sys: ks isp pp/ep/e}_{0,0}$ 
exhibits a competitive mechanism between diffusion and chemotaxis-driven aggregation. 
Particularly, blowup---the extreme form of aggregation---may occur in finite time under certain conditions: 
\begin{itemize} 
\item No solution of $\eqref{sys: ks isp pp/ep/e}_{0,0}$ blows up in one dimension~\cite{Osaki2001}. 
\item In a planar domain, blowup may occur in finite or infinite time if the initial mass $m := \int_\Omega u_0$ exceeds the critical mass $4\pi$ but not equals to an integer multiple of $4\pi$ \cite{Horstmann2001}; all solutions remain globally bounded if $m < 4\pi$ (or $8\pi$ with symmetry assumptions) \cite{Nagai1997}. 
In a disk, there exist radially symmetric initial data with mass $m\in(8\pi,\infty)$ such that the solutions blow up in finite time \cite{Mizoguchi2014}.
\item In a ball of higher dimensions, for any prescribed initial mass $m > 0$, there exist radially symmetric initial data $(u_0,v_0)$ with mass $m=\int_\Omega u_0$ that lead to finite-time blowup~\cite{Winkler2013}.
\end{itemize}

Due to strong coupling of fully parabolic-type cross-diffusion, detecting blowup tightly relies on the energy dissipation structure of $\eqref{sys: ks isp pp/ep/e}_{\tau,\varepsilon}$, i.e., the identity (cf. \cite[Lemma~3.3]{Nagai1997}, \cite{Laurencot2019} and \cite[Proposition~6.1]{Fujie2017})
\begin{equation}
  \label{eq: dfcdt}
  \mathcal{F}'_{\tau,\varepsilon} = - \mathcal{D}_{\tau,\varepsilon}, 
  \quad \text{for all } t\in(0,T_{\max})
\end{equation} 
holds along the solution curve $\{(u(\cdot, t), v(\cdot, t), \tau w(\cdot, t))\}$,
where the Lyapunov functional $\mathcal{F}_{\tau,\varepsilon}(t) = \mathcal{F}_{\tau,\varepsilon}(u,v,w) := \mathcal{F}_{\tau,\varepsilon}(u,v)$ and its dissipation rate $\mathcal{D}_{\tau,\varepsilon}(t) = \mathcal{D}_{\tau,\varepsilon}(u,v)$ are given as follows,
\begin{equation}
  \label{sym: energy and dissipation rate}
  \begin{aligned}
    \mathcal{F}_{\tau,\varepsilon}(u,v) 
    &:= \int_\Omega u\ln u 
    - \int_\Omega uv 
    + \frac{\tau}{2}\int_\Omega v_t^2
    + \frac{\varepsilon}{2}\int_\Omega |\Delta v|^2
    + \frac{1+\varepsilon}{2}\int_\Omega |\nabla v|^2 
    + \frac{1}{2}\int_\Omega v^2\\
    &\;= \int_\Omega u\ln u 
    - \int_\Omega uv 
    + \frac{\tau}{2}\int_\Omega |\Delta v - v + w|^2
    + \frac{1}{2}\int_\Omega (\varepsilon\Delta v -v)(\Delta v - v),\\
    \mathcal{D}_{\tau,\varepsilon}(u,v) 
    &:= \int_\Omega u|\nabla (\ln u - v)|^2 
    + (\tau+\varepsilon)\int_\Omega |\nabla v_t|^2
    + (\tau+1)\int_\Omega v_t^2.
  \end{aligned}
\end{equation}
Horstmann and Wang~\cite{Horstmann2001} proposed to link the initial energy $\mathcal{F}_{\tau,\varepsilon}(u_0,v_0)$ and stationary energy $\mathcal{F}_{\tau,\varepsilon}(u_\infty,v_\infty)$ via dissipation. 
By absurdum, finite- or infinite-time blowup enforced by low initial energy can be inferred under conditions that stationary energies have lower bounds, and that initial data with arbitrarily large negative energy can be constructed. This idea was applied in~\cite{Horstmann2001,Winkler2010b} for $\eqref{sys: ks isp pp/ep/e}_{0,0}$, \cite{Laurencot2019} for $\eqref{sys: ks isp pp/ep/e}_{1,0}$ and \cite{Fujie2019,Mao2024a} for $\eqref{sys: ks isp pp/ep/e}_{1,1}$. 

Concerning finite-time blowup, Winkler~\cite{Winkler2013} originally paved the way to link the energy to its dissipation rate in a refined form 
\begin{equation}
  \label{eq: functional inequality}
  - \mathcal{F}_{0,0} \leq \int_\Omega uv + |\Omega|/e \leq C\mathcal{D}_{0,0}^{\theta} + C
  \quad \text{for some } C>0 \text{ and } \theta\in(0,1)
\end{equation}
for a large class of radially symmetric function pairs $(u,v)=(u(r),v(r))$, which, in particular, includes the radial trajectories of the system~$\eqref{sys: ks isp pp/ep/e}_{0,0}$.  
This result leads to a superlinear differential inequality for the quantity $-\mathcal{F}_{0,0}$ by reversing the inequality~\eqref{eq: functional inequality} in case of low initial energy $\mathcal{F}_{0,0}(u_0,v_0)$,
which, in turn, implies occurrence of finite-time blowup. 
Functional inequalities of the form \eqref{eq: functional inequality} have been established in various chemotaxis models involving volume-filling~\cite{Cieslak2012,Cieslak2014,Cieslak2015,Laurencot2017,Hashira2018,Cao2025}, competing chemotaxis~\cite{Lankeit2021}, the Cauchy problem~\cite{Winkler2020a}, two-species chemotaxis~\cite{Li2014,Ha2024} and indirect chemotaxis~$\eqref{sys: ks isp pp/ep/e}_{0,1}$~\cite{Mao2024b}.

\subsection*{Indirect chemotaxis}

In the pioneer work~\cite{Tao2017}, Tao and Winkler pointed out that the indirect attractant production mechanism revealed in the J\"ager-Luckhaus variant of the beetles model~$\eqref{sys: ks isp pp/ep/e}_{1,0}$ where the second equation was modified to $\Delta v - \int_\Omega w/|\Omega| + w = 0$, has a temporal relaxation effect that prevents occurrence of finite-time blowup under two-dimensional and symmetry settings. They showed the variant admits a global classical solution and exhibits a novel infinite-time blowup critical mass $8\pi$ phenomenon. Lauren\c{c}ot~\cite{Laurencot2019} constructed the Lyapunov functional $\mathcal{F}_{1,0}$
of the parabolic-parabolic-ODE system~$\eqref{sys: ks isp pp/ep/e}_{1,0}$ and extended the results in \cite{Tao2017} to an arbitrary domain $\Omega\subset\mathbb{R}^2$.

For fully parabolic indirect-chemotaxis models \eqref{sys: ks isp pp/ep/e}, 
the indirect attractant mechanism can also reduce explosion-supporting potential due to strong parabolic smoothing effect.
Fujie and Senba~\cite{Fujie2017,Fujie2019} established global boundedness in physical spaces $n\leq3$, and constructed the Lyapunov functional $\mathcal{F}_{1,1}$ to 
identify a four-dimensional critical mass $64\pi^2$ phenomenon for classical solutions of \eqref{sys: ks isp pp/ep/e}:
\begin{itemize}
  \item  
  Under symmetry assumptions, $\int_\Omega u_0 < 64\pi^2$ implies the solution exists globally and remains bounded. 
  Without radial symmetry, the same conclusion holds for \eqref{sys: ks isp pp/ep/e} under certain mixed boundary conditions. 
  \item For any $m \in(64\pi^2,\infty)\setminus 64\pi^2\mathbb N$, there exist initial data satisfying 
  $\int_\Omega u_0 = m$ such that 
  solutions of \eqref{sys: ks isp pp/ep/e} under mixed boundary conditions blow up in finite or infinite time.
\end{itemize}
In five and higher dimensions, Mao and Li~\cite{Mao2024a} showed that for any $m>0$, there exist initial data with $\int_\Omega u_0 = m$ such that solutions of the system \eqref{sys: ks isp pp/ep/e} blow up in finite or infinite time with symmetry assumptions.

Results on finite-time singularity formation of the system \eqref{sys: ks isp pp/ep/e} are limited to certain simplified versions of elliptic type. 
Tao and Winkler~\cite{Tao2025} considered the J\"ager-Luckhaus variant and finite-time blowup was detected via robust comparison methods, 
which are applicable to cooperative systems incorporating volume-filling~\cite{Painter2002} and phenotype switching~\cite{Macfarlane2022}.   
Mao and Li~\cite{Mao2024b} analyzed the Nagai-type variant $\eqref{sys: ks isp pp/ep/e}_{0,1}$ and deduced finite-time blowup in five and higher dimensions via a functional inequality of the form \eqref{eq: functional inequality}. 

Some variants of indirect chemotaxis involving growth~\cite{Hu2016}, volume-filling~\cite{Ding2019}, phenotype switching~\cite{Laurencot2021,Painter2023,Laurencot2024}, haptotaxis~\cite{Chen2024}, the Cauchy problem~\cite{Xiang2025,Hosono2025} and rotation~\cite{Chen2025} were studied.

\subsection*{Main ideas and results}
Even though the functional inequality of the form \eqref{eq: functional inequality} has been applied to the Nagai-type variant $\eqref{sys: ks isp pp/ep/e}_{0,1}$, 
it may not work for the fully parabolic version \eqref{sys: ks isp pp/ep/e},
due to the problematic term $w_t$ involving the so-called distributed delayed diffusion~\cite{Shi2021}. Technically, we cannot expect that $w_t$ has certain uniform-in-time estimates even in $L^1$ because of $u\in L^\infty((0,T_{\max});L^1)$ merely. On the other hand, $w_t$ is unlike $v_t$ that can be suitably controlled by quantities from either $\mathcal{F}_{\tau,\varepsilon}$ or $\mathcal{D}_{\tau,\varepsilon}$.
More or less the memory effect~\cite{Shi2021} encourages us in construction of an integral-in-time inequality instead. Following the simple and far-reaching ideas of linking energy to its dissipation rate~\cite{Winkler2013}, we show ``memory" of the cross term $\int_\Omega uv$ has an upper bound of a sublinear power of itself in the form of 
\begin{equation} 
  \label{eq: integral inequality}
  \begin{aligned}
  \int_0^s\int_\Omega uv\dd x\dd t 
  &\leq C(1+s) \left(\int_0^s\mathcal{D}_{1,1}(t)\dd t + \frac{1}{2}\int_\Omega |\Delta v - v|^2 
  + \frac{1}{2}\int_\Omega v_t^2  + 1\right)^\theta \\
  &= C(1+s) \left(\mathcal{F}_{1,1}(0) + \int_\Omega uv - \int_\Omega u\ln u + 1\right)^\theta 
  \end{aligned}
  \end{equation}
valid along the radial trajectories of the system~\eqref{sys: ks isp pp/ep/e},
  where $\theta \in(0,1)$ and $C>0$. 
This inequality leads to a superlinear differential inequality for the function 
\begin{equation*}
  \Psi (s) := \int_0^s\int_\Omega u(v-\ln u)\dd x\dd t + \mathcal{F}_{1,1}(0)s + \ell s
  \quad \text{for } s\in(0,T_{\max})
\end{equation*}
with some large value $\ell > 1$ provided that $-\mathcal{F}_{1,1}(0)$ is sufficiently large. 
This, in turn, implies occurrence of finite-time blowup for initial data with low energy $\mathcal{F}_{1,1}(0)$. 

To be precise, we shall suppose that initial data $(u_0, v_0, w_0)$ consist of nonnegative and radially symmetric functions satisfying 
\begin{equation}
  \label{h: initial data}
  \begin{cases}
  0\leq u_0\in C^0(\overline{\Omega}),
  \quad 0\leq v_0,w_0\in C^2(\overline{\Omega})
  \quad \text{with } \partial_\nu v_0 = \partial_\nu w_0 \equiv 0 \text{ on } \partial\Omega\\
  \text{and } (u_0,v_0,w_0) \text{ is a triplet of radially symmetric functions on } \overline{\Omega},
  \end{cases} 
\end{equation}
with $\Omega = B_R := \{x\in\mathbb{R}^n\mid |x| < R\}$ for some $R>0$ and $n\in\mathbb{N}$.

Our main result states that low initial energy enforces finite-time blowup.

\begin{theorem}
  \label{thm: low energy enforced finite-time blowup}
  Assume $\Omega = B_R \subset\mathbb R^n$ for some $n\geq5$ and $R>0$.
  Let $\tilde{m} > 0$ and $A > 0$ be given.
  Then there exist positive constants $K(\tilde{m},A)$ and $T(\tilde{m},A)$ such that 
  if $(u_0,v_0,w_0)$ from the set 
  \begin{equation}
    \label{sym: mathcal B}
    \begin{aligned}
       \mathcal{B}(\tilde{m},A) &:= \{(u_0,v_0,w_0)\in  C^0(\overline{\Omega}) \times C^2(\overline{\Omega}) \times C^2(\overline{\Omega})
       \mid 
        u_0, v_0 \text{ and } v_0 \\ 
        & \text{ comply with } \eqref{h: initial data},   
        \|u_0\|_{L^1(\Omega)} < \tilde{m} \text{ and }
        \|w_0\|_{W^{1,2}(\Omega)} + \|v_0\|_{W^{2,2}(\Omega)} < A\},
    \end{aligned}
  \end{equation}
  satisfies $\mathcal{F}_{1,1}(u_0,v_0,w_0) < - K(\tilde{m},A)$,
  then the corresponding classical solution 
  given by Proposition~\ref{prop: local existence and uniqueness} blows up in finite time $T_{\max}(u_0,v_0,w_0) \leq T(\tilde{m},A)$.
  Here, $\mathcal F_{1,1}$ is defined in \eqref{sym: energy and dissipation rate}.
\end{theorem}

Concerning the existence of initial data with low energy,
for any given nonnegative initial datum with suitable regularity,
we construct a family of initial data with arbitrarily large negative energy,
which can approximate the given initial datum in an appropriate topology. 

\begin{theorem}
  \label{thm: dense initial data for blowup}
  Assume $\Omega = B_R \subset\mathbb R^n$ for some $n\geq5$ and $R>0$.
  For any triplet $(u_0, v_0, w_0)$ of nonnegative and radially symmetric functions satisfying \eqref{h: initial data},
  one can find a family
  \begin{displaymath}
    \{(u_\eta, v_\eta, w_\eta) \in L^\infty(\Omega) \times W^{2,\infty}(\Omega) \times W^{2,\infty}(\Omega)\}_{\eta\in(0,1)}
  \end{displaymath}
  of function triplets with properties:
  $(u_\eta, v_\eta, w_\eta)$ complies with \eqref{h: initial data} for all $\eta\in(0,1/2)\cap(0,R)$ such that 
  \begin{equation}
    \label{eq: uetatou0}
    u_\eta \to u_0 \quad \text{in } L^p(\Omega) 
    \text{ for all } p\in\left[1,\frac{2n}{n+4}\right)
    \text{ as } \eta\searrow0,
  \end{equation}
  and 
  \begin{equation}
    \label{eq: vetatov0}
    v_\eta \to v_0
    \text{ and } 
    w_\eta \to w_0 
    \quad \text{in } W^{2, 2}(\Omega) 
    \text{ as } \eta\searrow0,
  \end{equation}
  but 
  \begin{equation}
    \label{eq: mathcalFuetaveta}
    \mathcal F_{1,1}(u_\eta, v_\eta, w_\eta)\to-\infty\quad \text{as } \eta\searrow 0.
  \end{equation}
  In particular, 
  \begin{displaymath}
    \inf_{(u,v,w)\in \mathcal B(\tilde{m}, A)}\mathcal{F}_{1,1}(u,v,w) = - \infty
  \end{displaymath}
  holds for all $\tilde{m} > 0$ and $A > 0$. 
  Here, $\mathcal F_{1,1}$ is taken from \eqref{sym: energy and dissipation rate}.
\end{theorem}

\begin{remark}
  \label{r: precedents} 
  The inequality~\eqref{eq: integral inequality} may shed light on construction of certain integral inequality for the indirect chemotaxis of parabolic-parabolic-ODE type $\eqref{sys: ks isp pp/ep/e}_{1,0}$, where $w$ and consequently $v$ have much lower regularity. Notably, the functional inequality of the form \eqref{eq: functional inequality} works for both classical Keller-Segel model~$\eqref{sys: ks isp pp/ep/e}_{0,0}$~\cite{Winkler2013} and the Nagai-type variant~$\eqref{sys: ks isp pp/ep/e}_{0,1}$~\cite{Mao2024b}, both of which have no temporal relaxation mechanism $\tau=0$.  

  The primary challenge lies in construction of the integral inequality \eqref{eq: integral inequality} due to temporal-relaxation effect of indirect attractant mechanisms. 
  Mathematically, we shall obtain spatial-temporal estimates $\int_0^s\|\Delta v\|_{L^2(B_\rho)}^2\dd t$ for all $(\rho,s)\in(0,R)\times(0,T_{\max})$ in terms of both energy $\mathcal{F}_{1,1}$ and its dissipation rate $\mathcal{D}_{1,1}$, except the term $\int_\Omega u\ln u$ from $\mathcal{F}_{1,1}$.  
  We are stimulated by \cite{Winkler2013,Cabre2020} to introduce the Poho\v{z}aev-type test function $\mathds{1}_{B_\rho}|x|^{n-2}x\cdot\nabla v$ to obtain the second-order estimate. Here, $\mathds{1}_E$ stands for the indicator function of a Borel set $E\subseteq\mathbb R^n$.
  Formally speaking, we implement integration by parts three times to transfer the operator $\partial_t - r^{1-n}(\partial_{r}r^{n-1})_r$ from acting on $w$ to acting on $v$.
\end{remark}

This paper is organized as follows. In Section~\ref{section preliminary},
we present some preliminaries.
In Section~\ref{sec: linking energy to dissipation}, we establish the functional inequality \eqref{eq: integral inequality}.
Section~\ref{sec: finite-time blowup} and Section~\ref{sec: initial data} are devoted to the proof of Theorem~\ref{thm: low energy enforced finite-time blowup} and Theorem~\ref{thm: dense initial data for blowup}, respectively.

\section{Preliminaries. Energy functional and basic estimates}
\label{section preliminary}

We first recall local-in-time existence and uniqueness of classical solutions to the system \eqref{sys: ks isp pp/ep/e},
which has been established in~\cite[Section~4]{Fujie2017}.

\begin{proposition}
  \label{prop: local existence and uniqueness}
 Let $\Omega = B_R\subset\mathbb{R}^n$ for some $R > 0$ and $n\in\mathbb N$. 
 Assume that $(u_0, v_0, w_0)$ is as in \eqref{h: initial data}.
 Then there exist $T_{\max } = T_{\max}(u_0, v_0, w_0) \in(0, \infty]$ and a unique triplet $(u, v, w)$ of nonnegative and radially symmetric functions from 
 $C^0\left(\bar{\Omega} \times\left[0, T_{\max }\right)\right) \cap C^{2,1}\left(\bar{\Omega} \times\left(0, T_{\max }\right)\right)$ 
 that solves \eqref{sys: ks isp pp/ep/e} classically in $\Omega \times\left(0, T_{\max }\right)$.  
 Also, the solution $(u,v,w)$ satisfies $u>0$, $v>0$ and $w>0$ in $\Omega \times\left(0, T_{\max }\right)$ and 
 \begin{equation} 
 \text{if } T_{\max } < \infty,
  \text{ then }
  \|u(\cdot, t)\|_{L^{\infty}(\Omega)}\to\infty 
  \text{ as } t \nearrow T_{\max }.
 \end{equation}
\end{proposition}

The Lyapunov functional of the system \eqref{sys: ks isp pp/ep/e} was constructed in~\cite[Proposition~6.1]{Fujie2017}. In the sequel, we may abuse the notation $\mathcal{F}(t) = \mathcal{F}(u,v) = \mathcal{F}(u,v,w)$.

\begin{proposition}
  Let $(u,v,w)$ be a classical solution given as in Proposition~\ref{prop: local existence and uniqueness}. 
  Then the following identity holds:
\begin{equation} 
  \label{eq: energyequation}
\frac{d}{d t} \mathcal{F}(u(t), v(t))+\mathcal{D}(u(t), v(t))=0 \quad \text {for all } t \in(0, T_{\max}),
\end{equation}
where
\begin{equation}
  \label{sym: mathcalFD}
\begin{aligned}
 \mathcal{F}(u, v)&=\int_{\Omega}(u \log u- u v)
+ \frac{1}{2}\int_\Omega |v_t|^2
+\frac{1}{2} \int_{\Omega}|(-\Delta+1) v|^2\\
&= \int_{\Omega}(u \log u- u v)
+ \frac{1}{2}\int_\Omega |\Delta v - v + w|^2
+\frac{1}{2} \int_{\Omega}|(-\Delta+1) v|^2
\quad\text{and} \\
\mathcal{D}(u, v) &=
2\int_{\Omega}\left(\left|\nabla v_t\right|^2+\left|v_t\right|^2\right)
+\int_{\Omega} u|\nabla(\log u - v)|^2 \\
&= 2\int_{\Omega}\left(\left|\nabla(\Delta v - v + w)\right|^2+\left|\Delta v - v + w\right|^2\right)
+\int_{\Omega} u|\nabla(\log u - v)|^2 .
\end{aligned}
\end{equation}
\end{proposition}

For the sake of simplicity in notation, 
we shall collect some time-independent estimates specifically for high dimensions, i.e., $n\geq4$. These estimates will serve as a starting point for obtaining certain pointwise estimates. 

\begin{lemma}
  \label{le: basic estimates}
  Let $n\geq4$ and 
  $(u,v,w)$ be a classical solution given as in Proposition~\ref{prop: local existence and uniqueness}.
  Then 
  \begin{equation}
    \label{eq: mass identity}
    \int_{\Omega} u(x, t) =\int_{\Omega} u_{0} =: m
      \quad\text{for all } t \in (0, T_{\max})
  \end{equation}
  \begin{equation}
    \label{eq: w-L1estimate}
    \int_{\Omega} w(x, t) \leq \max\left\{\int_\Omega u_0, \int_\Omega w_0\right\}
      \quad\text{for all } t \in (0, T_{\max})
  \end{equation}
  and 
  \begin{equation}
    \label{eq: vmassinequality}
    \int_\Omega v(x,t) 
    \leq \max\left\{\int_\Omega u_0, \int_\Omega v_0, \int_\Omega w_0\right\}
    \quad \text{for all } t\in(0,T_{\max}).
  \end{equation}
  For each 
  \begin{equation*}
    q\in\left[1,\frac{n}{n-1}\right),
  \end{equation*} 
  there exists $C_q>0$ such that 
  \begin{equation}
    \label{eq: wW1q}
    \|w\|_{W^{1,q}(\Omega)} \leq C_q(m + \|w_0\|_{W^{1,2}(\Omega)})
    \quad\text{for all } t \in (0, T_{\max}).
  \end{equation}
  For each 
  \begin{equation*}
    p\in\left[1,\frac{n}{n-2}\right),
  \end{equation*} 
  there exists $C_p > 0$ such that 
  \begin{equation}
    \label{eq: vW2p}
    \|v_t\|_{L^p(\Omega)}
    + \|w\|_{L^p(\Omega)}
    + \|v\|_{W^{2,p}(\Omega)} 
    \leq C_p(m + \|w_0\|_{W^{1,2}(\Omega)} + \|v_0\|_{W^{2,2}(\Omega)})
  \end{equation}
  for all $t \in (0, T_{\max})$.
\end{lemma}

\begin{proof}
The mass identity \eqref{eq: mass identity} follows directly 
from integrating the first equations in \eqref{sys: ks isp pp/ep/e} over $\Omega$. 
To see \eqref{eq: w-L1estimate}, we integrate the last equation and get 
\begin{equation*}
  \int_\Omega w(x,t) 
  = e^{-t}\int_\Omega w_0  
  + (1-e^{-t})\int_\Omega u_0 
  \leq \max\left\{\int_\Omega w_0, \int_\Omega u_0\right\}
  \quad \text{for all } t\in(0,T_{\max}).
\end{equation*}
Integrating the second equation yields  
\begin{align*}
  \int_\Omega v(x,t) 
  &= e^{-t}\int_\Omega v_0  
  + \int_0^te^{s-t}\int_\Omega w(x,s)\dd x\dd s \\
  &\leq e^{-t}\int_\Omega v_0  
  + (1-e^{-t}) \max\left\{\int_\Omega w_0, \int_\Omega u_0\right\}\\
  &\leq \max\left\{\int_\Omega u_0, \int_\Omega v_0, \int_\Omega w_0\right\}
  \quad \text{for all } t\in(0,T_{\max}).
\end{align*}
This verifies \eqref{eq: vmassinequality}.
By applying well-known Neumann heat semigroup estimates \cite[Lemma~1.3]{Winkler2010} to 
the parabolic initial-boundary value problem $w_t -\Delta w + w = u$, 
we can verify \eqref{eq: wW1q} as done in \cite[Lemma~3.1]{Winkler2013}.
Thus, $w\in L^p(\Omega)$ for $p\in[1,n/(n-2))$ by the Sobolev embedding theorem. 
Moreover, for any $p\in[1,n/(n-2))$ one can select 
$q\in[1,n/(n-1))$ such that 
\begin{equation}
  \label{eq: pqchoices}
  -\frac{1}{2}-\frac{n}{2}\left(\frac{1}{q}-\frac{1}{p}\right) > -1.
\end{equation}
$L^p$--$L^q$ estimates of the heat semigroup $e^{t\Delta}$ 
with homogeneous Dirichlet boundary conditions (cf. \cite[Lemma~3.1 and 3.2]{Ma2024}) 
applied to $v_{x_it} = \Delta v_{x_i} - v_{x_i} + w_{x_i}$ 
subject to $v_{x_i} = v_rx_i/r \equiv 0$ on $\partial\Omega$ for $i \in\{1,2,\ldots,n\}$,
provide a constant $C>0$ such that 
  \begin{align*}
    \|\nabla v_{x_i}\|_{L^p(\Omega)} 
    &\leq C\|v_{0x_i}\|_{W^{1,p}(\Omega)} \\
    &\quad + C\int_0^t\left(1+(t-s)^{-\frac{1}{2}-\frac{n}{2}\left(\frac{1}{q}-\frac{1}{p}\right)}\right)e^{-(t-s)}\dd s 
    \sup_{s\in(0,t)}\|w_{x_i}(\cdot,s)\|_{L^q(\Omega)} \\
    &\leq C\|v_0\|_{W^{2,2}(\Omega)} 
    + CC_q (m + \|w_0\|_{W^{1,2}(\Omega)}) \left(1+\Gamma\left(\frac{1}{2}-\frac{n}{2}\left(\frac{1}{q}-\frac{1}{p}\right)\right)\right),
  \end{align*}
holds for all $t\in(0,T_{\max})$, $p\in[1,n/(n-2))$ and $i \in\{1,2,\ldots,n\}$,
where the Euler's Gamma function $\Gamma$ is well-defined due to the choices of parameters in~\eqref{eq: pqchoices}. According to \eqref{eq: vmassinequality},
this implies that $v\in W^{2,p}(\Omega)$ for all $t\in(0,T_{\max})$, 
as well as \eqref{eq: vW2p},
since $v_t = \Delta v - v + w \in L^p(\Omega)$ for all $t\in(0,T_{\max})$.
\end{proof}

Thanks to symmetry assumptions, uniform-in-time pointwise estimates for $v$ and $w$ can be deduced from Lemma~\ref{le: basic estimates}.

\begin{lemma}
\label{le: pointwise estimate}
  Let $\Omega = B_R\subset\mathbb R^n$ with some $R > 0$ and $n\geq4$.
  If $(u,v,w)$ is a classical solution of \eqref{sys: ks isp pp/ep/e} 
  associated with radially symmetric initial datum $(u_0,v_0,w_0)$ as
  given in Proposition~\ref{prop: local existence and uniqueness}, 
  then for each $\beta > n - 2$, there exists a constant $C_\beta > 0$ such that 
  \begin{equation}
    \label{eq: wpointwise}
    \||x|^{\beta}w\|_{L^\infty(\Omega)} 
    \leq C_\beta(m + \|w_0\|_{W^{1,2}(\Omega)})
  \end{equation}
  for all $t\in(0,T_{\max})$,
  and that 
  \begin{equation}
    \label{eq: vpointwise}
    \||x|^{\beta-1} v\|_{W^{1,\infty}(\Omega)}
    \leq C_\beta(m + \|w_0\|_{W^{1,2}(\Omega)} + \|v_0\|_{W^{2,2}(\Omega)})
  \end{equation}
  for all $t\in(0,T_{\max})$.
\end{lemma}

\begin{proof}
  The weighted uniform estimates \eqref{eq: wpointwise} and \eqref{eq: vpointwise} are direct consequences of uniform-in-time Sobolev regularity estimates \eqref{eq: wW1q} and \eqref{eq: vW2p}, since the Sobolev spaces of radial functions can be embedded into certain weighted Sobolev spaces~\cite{GuedesdeFigueiredo2011}. Moreover, detailed proofs of \eqref{eq: wpointwise} and \eqref{eq: vpointwise} can be found in \cite[Lemma~3.2]{Winkler2013} and \cite[Lemma~3.3]{Mao2024b}, respectively.
\end{proof}

\section{Construction of an integral inequality}
\label{sec: linking energy to dissipation}

The goal of this section is to establish the following integral inequality as aforementioned in the introduction.
\begin{proposition}
  \label{prop: an integral inequality}
  Let $\Omega = B_R\subset\mathbb{R}^n$ for some $n\geq 5$ and $R>0$. Assume $\tilde{m} > 0$ and $A>0$. Then there exists $C>0$ and $\theta \in(0,1)$ such that for any $(u_0,v_0,w_0)$ from the set $\mathcal{B}(\tilde{m},A)$, the inequality 
  \begin{equation}
    \label{eq: integral inequality in prop} 
    \int_0^s\int_\Omega uv\dd x\dd t 
    \leqslant C(\tilde{m}+A+1)^2(1+s) \left(\mathcal{F}(0) + \int_\Omega uv - \int_\Omega u\ln u + 1\right)^\theta 
    \end{equation}
    holds for all  $s\in(0,T_{\max})$,
    where $(u,v,w)$ and $T_{\max}$ are given as in Proposition~\ref{prop: local existence and uniqueness}, and $\mathcal{F}(0):=\mathcal{F}(u_0,v_0,w_0)$ is defined in \eqref{sym: mathcalFD}.
\end{proposition}

\subsection{Notations and setup}
For convenience, we shall switch to the radial notation without any further comment, e.g., writing $v (r, t)$ instead of $v (x, t)$ when appropriate. We denote $\omega_n$ the surface area of the unit ball in $\mathbb{R}^n$ and denote $C > 0$ the generic constant, which may vary from line to line.

From now on, we fix $n\geq5$, $R>0$, $\kappa > n - 2$, $\tilde{m} > 0$ and $A > 0$. 
Lemma~\ref{le: basic estimates} and Proposition~\ref{le: pointwise estimate} entails one can find a positive constant $C_0>0$ with the following properties:   
if initial datum $(u_0,v_0,w_0)$ satisfies \eqref{h: initial data} such that
\begin{equation}
  \label{eq: initial constraints}
  \int_\Omega u_0 = m < \tilde{m}
  \quad\text{and}\quad 
  \|w_0\|_{W^{1,2}(\Omega)} + \|v_0\|_{W^{2,2}(\Omega)} < A,
\end{equation}
i.e., $(u_0,v_0,w_0)\in\mathcal{B}(\tilde{m}, A)$,
then mass constraints 
\begin{equation}
  \label{eq: constraint of mass}
  \int_\Omega u + \int_\Omega v + \int_\Omega w < \tilde{m} + C_0A =:  M \quad\text{for all } t\in(0,T_{\max})
\end{equation}
and pointwise estimates
\begin{equation}
  \label{eq: pointwise estimates}
  \||x|^{\kappa}w\|_{L^{\infty}(\Omega)} + \||x|^{\kappa-1} v\|_{W^{1,\infty}(\Omega)} < C_0(\tilde{m} + A) =: B \quad\text{for all } t\in(0,T_{\max}),
\end{equation}
as well as first-order estimates
\begin{equation}
  \label{eq: first-order estimates}
  \|v_t\|_{L^1(\Omega)} < C_0(\tilde{m} + A) = B 
  \quad \text{for all } t\in(0,T_{\max})
\end{equation}
hold,
where the radial functions $(u,v,w)$ and $T_{\max} = T_{\max}(u_0,v_0,w_0)$ are given in Proposition~\ref{prop: local existence and uniqueness}.  

To construct the integral inequality~\eqref{eq: integral inequality in prop}, 
we introduce the following quantities 
\begin{equation}
  \label{sym: f}
  f := - \Delta v + v - w = -v_t,\quad (x,t)\in \Omega\times(0,T_{\max}), 
\end{equation}
and 
\begin{equation}
  \label{sym: g}
  g := \left(\frac{\nabla u}{\sqrt u} - \sqrt u \nabla v\right)\cdot \frac{x}{|x|},
  \quad x\in\Omega\setminus\{0\}, \quad t\in(0,T_{\max}).
\end{equation} 
Then the energy functional $\mathcal{F}$ and the dissipation rate $\mathcal{D}$ are given by 
\begin{equation} 
  \label{sym: mathacl F D}
\begin{aligned} 
& \mathcal{F}(t)=\int_{\Omega}(u \log u- u v)
+ \frac{1}{2} \int_\Omega |f|^2 
+\frac{1}{2} \int_{\Omega}|\Delta v - v|^2 \quad \text{and} \\
& \mathcal{D}(t)=
2\int_{\Omega}\left(\left|\nabla f\right|^2+\left|f\right|^2\right)
+\int_{\Omega} g^2 .
\end{aligned}
\end{equation}

\subsection{Interpolation inequalities for concentration \texorpdfstring{$v$}{v}}

We shall collect uniform-in-time estimates for concentration $v$, as direct consequences or variants of the Gagliardo-Nirenberg inequality~\cite{Nirenberg1959}, which will be frequently used. 

\begin{lemma}
  \label{le: GN}
  Assume that $(u_0,v_0,w_0)\in\mathcal{B}(\tilde{m}, A)$.
  Then there exists $C>0$ such that 
  \begin{equation}
    \label{eq: v2<deltav}
    \|v\|_{L^2(\Omega)}
\leq C M^{4/(n+4)}\|\Delta v - v\|_{L^2(\Omega)}^{n/(n+4)}
\quad \text{for all } t\in(0,T_{\max}),
  \end{equation}
  and 
  \begin{equation}
    \label{eq: nablav2<deltav} 
      \|\nabla v\|_{L^2(\Omega)}
      \leq CM^{2/(n+4)}\|\Delta v - v\|_{L^2(\Omega)}^{(n+2)/(n+4)}
      \quad \text{for all } t\in(0,T_{\max}),
  \end{equation}
  and 
  \begin{equation}
    \label{eq: f gradf gn}
    \|f\|_{L^2(\Omega)}^2 
    \leq CB^{4/(n+2)} \|f\|_{W^{1,2}(\Omega)}^{2n/(n+2)},
    \quad \text{for all } t\in(0,T_{\max}),
  \end{equation}
  as well as 
  \begin{equation}
    \label{eq:v2innergn}
    \|v\|_{L^2(\Omega)}^2 
    \leq \frac{1}{2}\|\nabla v\|_{L^2(B_\rho)}^2 
    + C(M^2 + B^2)\rho^{n-2\kappa}
    \quad \text{for all } (\rho, t)\in(0,R)\times(0,T_{\max}),   
  \end{equation}
  where the solution $(u,v,w)$ of the system~\eqref{sys: ks isp pp/ep/e} and the number $T_{\max}\in(0,\infty]$ are defined in Proposition~\ref{prop: local existence and uniqueness},
  and $f = -v_t$ is defined at \eqref{sym: f}. 
\end{lemma}

\begin{proof}
Gagliardo-Nirenberg inequality \cite{Nirenberg1959} provides a constant $C_{gn} > 0$ such that 
\begin{align}
  \label{eq: psi l2 grad} 
  \|\psi\|_{L^2(\Omega)}^2
  &\leq C_{gn} \|\nabla \psi\|_{L^2(\Omega)}^{2n/(n+2)}\|\psi\|_{L^1(\Omega)}^{4/(n+2)} 
  + C_{gn}\|\psi\|_{L^1(\Omega)}^2
  \quad \text{for all } \psi\in W^{1,2}(\Omega)
\end{align} 
and 
\begin{align}
  \label{eq: psi l2 W12} 
\|\psi\|_{L^2(\Omega)}^2
&\leq C_{gn} \|\psi\|_{W^{1,2}(\Omega)}^{2n/(n+2)}\|\psi\|_{L^1(\Omega)}^{4/(n+2)}
\quad \text{for all } \psi\in W^{1,2}(\Omega)
  \end{align} 
as well as  
\begin{align}
  \label{eq: varphi l2 delta} 
\|\varphi\|_{L^2(\Omega)}
&\leq C_{gn} \|\Delta \varphi - \varphi\|_{L^2(\Omega)}^{n/(n+4)}\|\varphi\|_{L^1(\Omega)}^{4/(n+4)}
\end{align} 
holds for all $\varphi\in W^{2,2}(\Omega)$ with $\partial\varphi/\partial\nu = 0$ on $\partial\Omega$.
Here, the elliptic $L^2$-regularity theory has been used, 
i.e., there exists $C_e > 0$ such that 
\begin{equation*}
\|\varphi\|_{W^{2,2}(\Omega)} \leq C_e \|\Delta \varphi - \varphi\|_{L^2(\Omega)}
\end{equation*} 
holds for all $\varphi\in W^{2,2}(\Omega)$ with $\partial\varphi/\partial\nu = 0$ on $\partial\Omega$.  
Therefore, \eqref{eq: v2<deltav} is a direct consequence of mass constraints~\eqref{eq: constraint of mass} and \eqref{eq: varphi l2 delta}. 
Invoking the first-order estimate \eqref{eq: first-order estimates}, we thus have \eqref{eq: f gradf gn} from \eqref{eq: psi l2 W12}.  
Using \eqref{eq: psi l2 grad} and mass constraints \eqref{eq: constraint of mass}, we estimate by Young inequality 
\begin{equation}
  \label{eq: v gradv gn}
  \begin{aligned} 
  \|v\|_{L^2(B_\rho)}^2 
  &\leq C_{gn}M^{4/(n+2)} \|\nabla v\|_{L^2(\Omega)}^{2n/(n+2)} 
  + C_{gn} M^2\\
  &\leq \frac{1}{2}\|\nabla v\|_{L^2(\Omega)}^2
  + 2^nC_{gn}^{(n+2)/2}M^2
  + C_{gn} M^2
  \end{aligned}
\end{equation}
for all $\rho\in(0,R)$ and $t\in(0,T_{\max})$.
Pointwise estimates \eqref{eq: pointwise estimates} entails 
\begin{equation}
  \label{eq: nablavouterregion}
  \begin{aligned} 
    \|\nabla v\|^2_{L^2(\Omega)}
    &= \|\nabla v\|_{L^2(B_\rho)}^2 
    + \|\nabla v\|^2_{L^2(\Omega\setminus B_\rho)}\\
    &\leq \|\nabla v\|_{L^2(B_\rho)}^2  
    +B^2\||x|^{1-\kappa}\|_{L^2(\Omega\setminus B_\rho)}^2\\
    &\leq \|\nabla v\|_{L^2(B_\rho)}^2 
    + \omega_nB^2R^2\rho^{n-2\kappa}
  \end{aligned}
\end{equation}
for all $\rho\in(0,R)$ and $t\in(0,T_{\max})$. 
Inserting \eqref{eq: nablavouterregion} into \eqref{eq: v gradv gn} yields \eqref{eq:v2innergn}.
Integrating by parts, using H\"older inequality and~\eqref{eq: v2<deltav}, we get 
\begin{align*} 
    \|\nabla v\|^2_{L^2(\Omega)} 
    + \|v\|_{L^2(\Omega)}^2 
    &= \int_\Omega (-\Delta v + v)v
    \leq \|\Delta v - v\|_{L^2(\Omega)}\|v\|_{L^2(\Omega)}\\
    &\leq CM^{4/(n+4)}\|\Delta v - v\|_{L^2(\Omega)}^{2(n+2)/(n+4)}
\end{align*}
which verifies \eqref{eq: nablav2<deltav}.
\end{proof}

\subsection{Spatial-temporal estimates of \texorpdfstring{$uv$}{uv}}

We incorporate decoupling techniques from Cao and Fuest~\cite{Cao2025},
which assist us to convert the cross integral $\int_\Omega uv$ into the $L^2_{\loc}$--estimate for the Laplacian of concentration $v$. 
Notably,
we shall exhaust the contribution of the second-order terms $\int_\Omega |\Delta v - v|^2$ and $\int_\Omega v_t^2$ in the energy functional $\mathcal{F}$ 
to compensate terms related to $w_t$. 
We remark that the mixed term $\int_\Omega uv$ alone can be dominated by some power of the dissipation rate for the three- and higher-dimensional classical Keller-Segel model~$\eqref{sys: ks isp pp/ep/e}_{0,0}$~\cite{Winkler2013} with the exception in two dimensions~\cite{Mizoguchi2014}, and even for the five- and higher-dimensional Nagai-type variant of indirect chemotaxis~\cite{Mao2024b}. 
Whether it can be explained as either technical flaw or aggregation-suppressing effect of indirect attractant production \cite{Hu2016,Tao2017}, is unknown to us.

\begin{lemma}
  \label{le: uvL1smallball}
  Let $(u_0,v_0,w_0)\in\mathcal{B}(\tilde{m}, A)$.
  There exists a constant $C > 0$ such that 
  \begin{equation}
    \label{eq: intuvleq}
    \begin{aligned}
      \int_0^s\int_{\Omega} uv \dd x \dd t
      &\leq 2\int_0^s\|\Delta v\|_{L^2(B_\rho)}^2 \dd t
      + 2\int_0^s\|v\|_{L^2(B_\rho)}^2 \dd t \\
      &\quad + C M^{4/(n+4)}\|\Delta v - v\|_{L^2(\Omega)}^{2(n+2)/(n+4)} 
      + C M^{4/(n+4)}\|f\|_{L^2(\Omega)}^{2(n+2)/(n+4)}\\ 
      &\quad + CB(m + B)\rho^{n-2\kappa}s
    \end{aligned}
  \end{equation}
  holds for all $\rho\in (0,R)$ and $s\in(0,T_{\max})$.
\end{lemma}

\begin{proof}
  Let $\rho\in (0,R)$ and $s\in(0,T_{\max})$.
  We introduce a cut-off function defined as 
  \begin{displaymath}
    \phi^{(\rho)}(x) := \phi(x/\rho), 
    \quad x\in \mathbb R^n,\quad  \rho\in(0,R),
  \end{displaymath} 
  where $\phi(x)\in C_0^\infty(\mathbb R^n)$ is a nonnegative and radially symmetric function such that 
  \begin{displaymath}
    \phi(x) = 
    \begin{cases}
      1, & \text{if } x\in B_{1/2},\\
      0, & \text{if } x\in\mathbb R^n\setminus B_1.
    \end{cases}
  \end{displaymath}
  It follows that  
  \begin{displaymath}
    \rho\|\nabla \phi^{(\rho)}\|_{L^\infty(\mathbb R^n)} 
    + \rho^2\|\Delta \phi^{(\rho)}\|_{L^\infty(\mathbb R^n)} 
    \leq C_\phi := 
    \|\nabla \phi\|_{L^\infty(\mathbb R^n)} 
    + \|\Delta \phi\|_{L^\infty(\mathbb R^n)}  
  \end{displaymath}
  for each $\rho\in(0,R)$.
  Testing $u = w_t - \Delta w + w$ by $v\phi^{(\rho)}$, 
  we compute upon integration by parts,  
  \begin{equation}
    \label{eq: uvpsi}
    \begin{aligned}
      \int_{B_{\rho/2}} uv &\leq \int_\Omega uv\phi^{(\rho)} 
      = \int_\Omega w_tv\phi^{(\rho)} - \int_\Omega \Delta w v\phi^{(\rho)} + \int_\Omega wv\phi^{(\rho)}\\
      &= \int_\Omega (wv)_t\phi^{(\rho)} 
      - \int_\Omega wv_t\phi^{(\rho)}
      + \int_\Omega w(-\Delta(v\phi^{(\rho)}) + v\phi^{(\rho)})\\ 
      &= \int_\Omega (wv)_t\phi^{(\rho)} 
      + \int_\Omega (-\Delta v + v - v_t)w\phi^{(\rho)} 
      - \int_\Omega w(2\nabla v\cdot\nabla\phi^{(\rho)} + v\Delta\phi^{(\rho)})
    \end{aligned}
  \end{equation}
  for all $t\in(0,T_{\max})$.
  Recalling the notation $w = -\Delta v + v - f$ specified in \eqref{sym: f},
  we estimate 
  \begin{equation}
    \begin{aligned}
      \int_\Omega (-\Delta v + v - v_t)w\phi^{(\rho)}  
      &= \int_\Omega (-\Delta v + v +f)(- \Delta v + v -f)\phi^{(\rho)} \\
      &= \int_\Omega (\Delta v - v)^2\phi^{(\rho)} 
       - \int_\Omega f^2\phi^{(\rho)}\\ 
      &\leq 2\|\Delta v\|_{L^2(B_\rho)}^2 
      + 2\|v\|_{L^2(B_\rho)}^2
    \end{aligned}
  \end{equation}
  for all $t\in(0,T_{\max})$.
  Using the pointwise information from \eqref{eq: pointwise estimates}, 
  we estimate 
  \begin{equation}
    \begin{aligned}
      &- \int_\Omega w(2\nabla v\cdot\nabla\phi^{(\rho)} + v\Delta\phi^{(\rho)}) \\
      &\leq 2B^2C_\phi \int_{B_\rho\setminus B_{\rho/2}}
      |x|^{-\kappa}(|x|^{1-\kappa}|x|^{-1} + |x|^{2-\kappa}|x|^{-2})\\
      &= 4B^2C_\phi\omega_n\int_{\rho/2}^\rho r^{n-1-2\kappa}\dd r 
      \leq 2^{2\kappa-n+2}\omega_nB^2C_\phi\rho^{n-2\kappa}
    \end{aligned}
  \end{equation}
  for all $t\in(0,T_{\max})$,
  and 
  \begin{equation}
    \label{eq: uvannulus}
    \begin{aligned}
    \int_{\Omega \backslash B_{\rho/2}} uv 
    &\leq \sup _{r \in\left(\rho/2, R\right)} v(r) \cdot \int_{\Omega \setminus B_{\rho/2}} u \\
    &\leq B \sup_{r \in\left(\rho/2, R\right)} r^{2-\kappa} \cdot \int_{\Omega} u 
    = 2^{\kappa-2} Bm \rho^{2-\kappa} 
    \leq 2^{\kappa-2}BmR^{\kappa + 2 - n}\rho^{n-2\kappa}
    \end{aligned}
    \end{equation}
  for all $t\in(0,T_{\max})$.
  Invoking \eqref{eq: v2<deltav}, we have by Young inequality 
  \begin{equation}
    \label{eq: int wv}
    \begin{aligned}
    \int_{\Omega}wv \phi^{\rho} 
    &= \int_{\Omega}(-\Delta v + v - f)v\phi^{\rho} \\
    &\leq 
    \|\Delta v - v\|_{L^2(\Omega)}\|v\|_{L^2(\Omega)}
    + \|f\|_{L^2(\Omega)}\|v\|_{L^2(\Omega)}\\ 
    &\leq C M^{4/(n+4)}\|\Delta v - v\|_{L^2(\Omega)}^{2(n+2)/(n+4)} 
    + C M^{4/(n+4)}\|\Delta v - v\|_{L^2(\Omega)}^{n/(n+4)}\|f\|_{L^2(\Omega)}\\
    &\leq C M^{4/(n+4)}\|\Delta v - v\|_{L^2(\Omega)}^{2(n+2)/(n+4)} 
    + C M^{4/(n+4)}\|f\|_{L^2(\Omega)}^{2(n+2)/(n+4)}
    \end{aligned}
  \end{equation}
  for all $t\in(0,T_{\max})$.
  Summing \eqref{eq: uvpsi}--\eqref{eq: uvannulus} and integrating over $(0,s)\subset(0,T_{\max})$ yields
  \begin{displaymath}
    \begin{aligned}
      \int_0^s\int_{\Omega} uv\dd x \dd t 
      &\leq \int_{\Omega} wv\phi^{\rho} - \int_{\Omega}w_0v_0\phi^{\rho}
      + 2\int_0^s\|\Delta v\|_{L^2(B_\rho)}^2 \dd t
      + 2\int_0^s\|v\|_{L^2(B_\rho)}^2 \dd t\\
      &\quad + 2^{2\kappa-n+2}\omega_nB^2C_\phi\rho^{n-2\kappa}s
      + 2^{\kappa-2}BmR^{\kappa + 2 - n}\rho^{n-2\kappa}s
    \end{aligned}
  \end{displaymath}
  which reduces to \eqref{eq: intuvleq} by inserting \eqref{eq: int wv}.
\end{proof}

\subsection{Spatial-temporal estimates of \texorpdfstring{$\Delta v$}{Delta v}}

This subsection is devoted to estimates of the quantity $\int_0^s\|\Delta v\|_{L^2(B_\rho)}^2 \dd s$ through a combination of terms from both the dissipation rate $\mathcal{D}$ and the energy functional $\mathcal{F}$.

We begin with an employment of the Poho\v{z}aev-type test function $|x|^{n-2}x\cdot\nabla v\mathds{1}_{B_\eta}$ with $\eta\in(0,R)$,
initially introduced in the seminal works of Winkler~\cite{Winkler2010b, Winkler2010,Winkler2013}. 
The first crucial aspect of this procedure lies in exploiting 
the structural information \eqref{sym: g}.

\begin{lemma}
  Let $(u_0,v_0,w_0)\in\mathcal{B}(\tilde{m}, A)$.
  Denote $\alpha = 2(n-1)$. Then 
  \begin{equation}
    \label{eq: 1/2y2etan-1}
    \begin{aligned}
      &\quad \frac{1}{2}y^2(\eta)\eta^{n-1}
      + w_r(\eta)v_r(\eta)\eta^{n-1}
      - \frac{n-1}{\eta^{n-1}}\int_0^\eta y^2r^{\alpha-1}\dd r \\
      &\leq  \frac{1}{\eta^{n-1}}\int_0^\eta (w v_r)_tr^\alpha\dd r 
      + \frac{\sqrt{m}}{\omega_n}\|g\|_{L^2(B_\eta)}
      + 2\omega_n^{-1}\|\nabla f\|_{L^2(B_\eta)}\|\Delta v\|_{L^2(B_\eta)} \\
      &\quad + \omega_n^{-1}\|f\|_{L^2(B_\eta)}\|\nabla v\|_{L^2(B_\eta)}
      + \omega_n^{-1}\|\nabla f\|_{L^2(B_\eta)}\|v\|_{L^2(B_\eta)}
      + \frac12 v^2(\eta)\eta^{n-1}\\
      &\quad + \frac{n-1}{\eta^{n-1}}\int_0^\eta f^2r^{\alpha-1}\dd r 
      + \frac{\alpha}{\eta^{n-1}}\int_0^\eta ur^{\alpha-1}\dd r
      - \frac12 f^2(\eta)\eta^{n-1}
      -u(\eta)\eta^{n-1}
    \end{aligned}
  \end{equation}
  for all $\eta\in(0,R)$ and $t\in(0,T_{\max})$.
  Here, abbreviate   
  \begin{displaymath}
    y := \Delta v 
    = v_{rr} + \frac{n-1}{r}v_r 
    = r^{1-n}(v_rr^{n-1})_r, 
    \quad (r,t)\in(0, R)\times(0,T_{\max})
  \end{displaymath}
  for brevity.
\end{lemma}

\begin{proof}
  Let $\eta\in(0,R)$ and $t\in(0,T_{\max})$.
  Multiplying $w_t - u + w = \Delta w$ by $|x|^{n-2}x\cdot\nabla v$ 
  and integrating over $B_\eta\subset\Omega$,
  we get 
  \begin{align}
    \label{eq: test by rn-1vr}
    \int_0^\eta w_t v_rr^\alpha\dd r 
    - \int_0^\eta uv_rr^\alpha\dd r 
    + \int_0^\eta wv_rr^\alpha\dd r
    &=  \int_0^\eta (w_rr^{n-1})_rv_rr^{n-1}\dd r.
  \end{align}
  Using the identity $w_tv_r = (wv_r)_t - wv_{rt}$ and the notation $f = -v_t$, we have 
  \begin{equation}
    \label{eq: wtvr}
    \begin{aligned}
      \int_0^\eta w_t v_rr^\alpha\dd r  
      &= \int_0^\eta(wv_r)_tr^\alpha\dd r + \int_0^\eta wf_rr^\alpha\dd r. 
    \end{aligned}
  \end{equation}
  Thank to $uv_r = u_r - g\sqrt{u}$ as given by \eqref{sym: g}, 
  we compute the second term on the left of \eqref{eq: test by rn-1vr}
  \begin{equation}
    \begin{aligned} 
    - \int_0^\eta uv_rr^\alpha\dd r 
    &= -\int_0^\eta (u_r-g\sqrt{u})r^\alpha \dd r \\ 
    &= -u(\eta)\eta^\alpha 
    + \alpha\int_0^\eta ur^{\alpha-1}\dd r 
    + \int_0^\eta g\sqrt{u}r^\alpha \dd r.
    \end{aligned}
  \end{equation}
  Integration by parts yields 
  \begin{equation} 
  \begin{aligned} 
    \int_0^\eta (w_rr^{n-1})_rv_rr^{n-1}\dd r
    &= w_r(\eta)v_r(\eta)\eta^\alpha 
    - \int_0^\eta w_rr^{n-1}(v_rr^{n-1})_r\dd r \\
    &= w_r(\eta)v_r(\eta)\eta^\alpha 
    - \int_0^\eta w_ryr^\alpha\dd r.
  \end{aligned} 
\end{equation}
  Using $w = - y + v -f$,
  we evaluate the last term on the left of \eqref{eq: test by rn-1vr}
  \begin{equation} 
    \label{eq: wvrr2n-2}
  \begin{aligned}
    \int_0^\eta wv_rr^\alpha\dd r 
    &= \int_0^\eta (-y + v -f)v_rr^\alpha\dd r \\
    &= - \int_0^\eta (v_rr^{n-1})_rv_rr^{n-1}\dd r 
    + \frac{1}{2}\int_0^\eta(v^2)_rr^\alpha\dd r
    - \int_0^\eta fv_rr^\alpha\dd r \\
    &= - \frac12v_r^2 (\eta)\eta^\alpha 
    + \frac12 v^2(\eta)\eta^\alpha 
    - (n-1)\int_0^\eta v^2r^{\alpha-1}\dd r 
    - \int_0^\eta fv_rr^\alpha\dd r 
  \end{aligned} 
  \end{equation}
  and the last term on the right of \eqref{eq: wtvr}
  \begin{equation}
    \begin{aligned}
      \int_0^\eta wf_rr^\alpha\dd r 
      &= - \int_0^\eta yf_rr^\alpha\dd r 
      + \int_0^\eta vf_rr^\alpha\dd r \\
      &\quad - \frac12 f^2(\eta)\eta^\alpha 
      + (n-1)\int_0^\eta f^2r^{\alpha-1}\dd r.
    \end{aligned}
  \end{equation}
  Using $w_r = - y_r + v_r - f_r$, we calculate 
  \begin{equation}
    \label{eq: wryr2n-2}
  \begin{aligned}
     - \int_0^\eta w_ryr^\alpha\dd r 
    &= - \int_0^\eta (- y_r + v_r - f_r)yr^\alpha\dd r \\
    &=  \int_0^\eta  yy_rr^\alpha\dd r 
    - \int_0^\eta (v_rr^{n-1})_r v_rr^{n-1}\dd r
    + \int_0^\eta  f_ryr^\alpha\dd r \\
    &= \frac{1}{2}y^2(\eta)\eta^\alpha 
    - (n-1)\int_0^\eta y^2r^{\alpha-1}\dd r 
    - \frac12v_r^2 (\eta)\eta^\alpha + \int_0^\eta f_ryr^\alpha\dd r.
  \end{aligned} 
  \end{equation}
  Collecting \eqref{eq: test by rn-1vr}--\eqref{eq: wryr2n-2} and  
  multiplying by $\eta^{1-n}$, 
  we get 
  \begin{align*}
      &\quad \frac{1}{2}y^2(\eta)\eta^{n-1}
      + w_r(\eta)v_r(\eta)\eta^{n-1}
      - \frac{n-1}{\eta^{n-1}}\int_0^\eta y^2r^{\alpha-1}\dd r \\
      &= \frac{1}{\eta^{n-1}}\int_0^\eta (w v_r)_tr^\alpha\dd r 
      + \frac{1}{\eta^{n-1}} \int_0^\eta g\sqrt{u}r^\alpha \dd r
      - \frac{2}{\eta^{n-1}}\int_0^\eta f_ryr^\alpha\dd r\\
      &\quad 
      + \frac{1}{\eta^{n-1}}\int_0^\eta (vf_r-fv_r)r^\alpha\dd r
      + \frac{n-1}{\eta^{n-1}}\int_0^\eta (f^2-v^2)r^{\alpha-1}\dd r \\
      &\quad -u(\eta)\eta^{n-1} 
      + \frac{\alpha}{\eta^{n-1}}\int_0^\eta ur^{\alpha-1}\dd r + \frac12 v^2(\eta)\eta^{n-1} 
      - \frac12 f^2(\eta)\eta^{n-1}    \\
      &\leq  \frac{1}{\eta^{n-1}}\int_0^\eta (w v_r)_tr^\alpha\dd r 
      + \frac{\sqrt{m}}{\omega_n}\|g\|_{L^2(B_\eta)}
      + 2\omega_n^{-1}\|\nabla f\|_{L^2(B_\eta)}\|\Delta v\|_{L^2(B_\eta)} \\
      &\quad + \omega_n^{-1}\|f\|_{L^2(B_\eta)}\|\nabla v\|_{L^2(B_\eta)}
      + \omega_n^{-1}\|\nabla f\|_{L^2(B_\eta)}\|v\|_{L^2(B_\eta)}
      + \frac12 v^2(\eta)\eta^{n-1}\\
      &\quad + \frac{n-1}{\eta^{n-1}}\int_0^\eta f^2r^{\alpha-1}\dd r 
      + \frac{\alpha}{\eta^{n-1}}\int_0^\eta ur^{\alpha-1}\dd r
      - \frac12 f^2(\eta)\eta^{n-1}
      -u(\eta)\eta^{n-1},
    \end{align*}
  where the last inequality follows from H\"older inequality.
  This is exactly \eqref{eq: 1/2y2etan-1}.
\end{proof}

A local-in-space integration of \eqref{eq: 1/2y2etan-1} over $B_\rho$ for $\rho\in(0,R)$ enables us to effectively establish an upper bound for the interior estimate $\|\Delta v\|_{L^2(B_\rho)}$ through a combination of dissipation rate quantities and a problematic term involving $w_t$, as the second advantage of the Poho\v{z}aev-type test procedure.

\begin{lemma}
  \label{le: 1/8+1/4}
  Let $(u_0,v_0,w_0)\in\mathcal{B}(\tilde{m}, A)$. 
  Denote $\alpha = 2(n-1)$. 
  Then there exists a constant $C > 0$ such that
  \begin{equation} 
    \label{eq: 1/8+1/4}
  \begin{aligned}
    &\quad \frac{1}{8} \|\Delta v\|_{L^2(B_\rho)}^2 
    + \frac{1}{4}\|\nabla v\|_{L^2(B_\rho)}^2 \\
    &\leq 49\rho^2\|\nabla f\|_{L^2(B_\rho)}^2 
    + C B^{4/(n+2)} \|f\|_{W^{1,2}(\Omega)}^{2n/(n+2)}
    + \sqrt{m}\rho\|g\|_{L^2(B_\rho)}\\
    &\quad + \omega_n\int_0^\rho \frac{1}{\eta^{n-1}}\int_0^\eta (w v_r)_tr^\alpha\dd r \dd\eta
    + 2m + C(B^2 + M^2)\rho^{n-2\kappa}
  \end{aligned} 
  \end{equation}
  holds for all $\rho\in (0,R)$ and $t\in(0,T_{\max})$.
\end{lemma}

\begin{proof}
  Let $\rho\in (0,R)$ and $t\in(0,T_{\max})$.
  Integrating \eqref{eq: 1/2y2etan-1} over $(0,\rho)\subset(0,R)$, we have 
  \begin{equation}
    \label{eq: 1/2inty2etan-1}
    \begin{aligned}
      &\quad \frac12\|\Delta v\|_{L^2(B_\rho)}^2 
      + \int_{B_\rho}\nabla w\nabla v\dd x 
      - \omega_n\int_0^\rho\frac{n-1}{\eta^{n-1}}\int_0^\eta y^2r^{\alpha-1}\dd r\dd\eta\\
      &\leq 
      \omega_n\int_0^\rho \frac{1}{\eta^{n-1}}\int_0^\eta (w v_r)_tr^\alpha\dd r \dd\eta
      + 2\rho\|\nabla f\|_{L^2(B_\rho)}\|\Delta v\|_{L^2(B_\rho)} 
      + \rho\sqrt{m}\|g\|_{L^2(B_\rho)}\\
      &\quad 
      + \rho \|f\|_{L^2(B_\rho)}\|\nabla v\|_{L^2(B_\rho)}
      + \rho \|v\|_{L^2(B_\rho)}\|\nabla f\|_{L^2(B_\rho)}
      + \frac12\|v\|_{L^2(B_\rho)}^2
      - \frac12\|f\|_{L^2(B_\rho)}^2 \\
      &\quad - \int_{B_\rho} u 
      + \omega_n\int_0^\rho\frac{n-1}{\eta^{n-1}}\int_0^\eta (f^2+2u)r^{\alpha-1}\dd r\dd\eta.
    \end{aligned}
  \end{equation}
  We calculate the second term on the left of \eqref{eq: 1/2inty2etan-1} by integration by parts, 
  the identity $w = -\Delta v + v -f$ from \eqref{sym: f} and Young inequality 
  \begin{equation} 
    \label{eq: wrvretan-1}
  \begin{aligned} 
    &\quad \int_{B_\rho}\nabla w\nabla v
    = \omega_nwv_r\rho^{n-1} 
    - \int_{B_\rho} w\Delta v \\
    &= \omega_nwv_r\rho^{n-1}
    + \|\Delta v\|^2_{L^2(B_\rho)} 
    - \int_{B_\rho} v\Delta v 
    + \int_{B_\rho} f\Delta v \\
    &\geq \omega_n(w-v)v_r\rho^{n-1}
    + \|\Delta v\|^2_{L^2(B_\rho)} 
    + \|\nabla v\|^2_{L^2(B_\rho)} 
    - \frac{1}{48}\|\Delta v\|_{L^2(B_\rho)}^2 
    - 12\|f\|_{L^2(B_\rho)}^2 
  \end{aligned} 
\end{equation}
for all $\rho\in(0,R)$ and $t\in(0,T_{\max})$.
The pointwise restrictions \eqref{eq: pointwise estimates} imply 
  \begin{equation}
    \omega_n (w-v)v_r\rho^{n-1}
    \geq - \omega_nB^2(\rho^{2-\kappa} + \rho^{-\kappa})\rho^{1-\kappa}\rho^{n-1} 
    = - \omega_nB^2(\rho^{2} + 1)\rho^{n-2\kappa}
  \end{equation}
  for all $\rho\in(0,R)$ and $t\in(0,T_{\max})$.
  By Fubini theorem, we compute the last terms on the both sides of \eqref{eq: 1/2inty2etan-1}
  \begin{equation}
    \begin{aligned}
      &- \omega_n\int_0^\rho\frac{n-1}{\eta^{n-1}}\int_0^\eta y^2r^{\alpha-1}\dd r\dd\eta 
      = - (n-1)\omega_n\int_0^\rho\int_r^\rho\frac{\dd\eta}{\eta^{n-1}} y^2r^{\alpha-1}\dd r\\
      = &- \frac{(n-1)\omega_n}{n-2}\int_0^\rho y^2r^{n-1}\left(1 - \frac{r^{n-2}}{\rho^{n-2}}\right)\dd r
      \geq - \frac{n-1}{n-2}\|\Delta v\|^2_{L^2(B_\rho)}
    \end{aligned}
  \end{equation}
  and 
  \begin{equation}
    \label{eq: f2+2u}
    \begin{aligned}
      \omega_n\int_0^\rho\frac{n-1}{\eta^{n-1}}\int_0^\eta (f^2+2u)r^{\alpha-1}\dd r\dd\eta 
      &\leq \frac{(n-1)\omega_n}{n-2}\int_0^\rho (f^2+2u)r^{n-1}\dd r \\
      &< \frac{3}{2}\|f\|_{L^2(B_\rho)}^2 
      + 3\int_{B_\rho}u.
    \end{aligned}
  \end{equation}
  Young inequality gives 
  \begin{equation}
    2\rho\|\nabla f\|_{L^2(B_\rho)}\|\Delta v\|_{L^2(B_\rho)}
    \leq \frac{1}{48} \|\Delta v\|_{L^2(B_\rho)}^2 
    + 48\rho^2\|\nabla f\|_{L^2(B_\rho)}^2,
  \end{equation}
  and 
  \begin{equation}
  \label{eq: nablavfyoung}
  \begin{aligned}
    &\quad \rho\|\nabla v\|_{L^2(B_\rho)}\|f\|_{L^2(B_\rho)} 
    + \rho\|v\|_{L^2(B_\rho)}\|\nabla f\|_{L^2(B_\rho)}\\
    &\leq \frac{1}{4}\|\nabla v\|_{L^2(B_\rho)}^2 
    + \rho^2\|f\|_{L^2(B_\rho)}^2 
    + \rho^2\|\nabla f\|_{L^2(B_\rho)}^2 
    + \frac{1}{4}\|v\|_{L^2(B_\rho)}^2.
  \end{aligned} 
  \end{equation}
  Substituting \eqref{eq: wrvretan-1}--\eqref{eq: nablavfyoung} into \eqref{eq: 1/2inty2etan-1} we get  
  \begin{equation} 
    \label{eq: 3/2-n-1/n-2}
  \begin{aligned}
    &\quad\left(\frac{3}{2} - \frac{1}{24} - \frac{n-1}{n-2}\right) \|\Delta v\|_{L^2(B_\rho)}^2 
    + \frac{3}{4}\|\nabla v\|_{L^2(B_\rho)}^2 \\
    &\leq 49\rho^2\|\nabla f\|_{L^2(B_\rho)}^2 
    + \sqrt{m}\rho\|g\|_{L^2(B_\rho)} 
    + \omega_n\int_0^\rho \frac{1}{\eta^{n-1}}\int_0^\eta (w v_r)_tr^\alpha\dd r \dd\eta\\
    &\quad + 2m 
    + \|v\|_{L^2(B_\rho)}^2 + (13+\rho^2)\|f\|_{L^2(B_\rho)}^2
    + \omega_nB^2(\rho^{2} + 1)\rho^{n-2\kappa}
  \end{aligned} 
  \end{equation}
  for all $\rho\in(0,R)$ and $t\in(0,T_{\max})$.
  Noting the fact 
  \begin{displaymath}
    \frac{3}{2} - \frac{1}{24} - \frac{n-1}{n-2} \geq \frac18, 
    \quad \text{as } n\geq5,
  \end{displaymath}
  we plug \eqref{eq:v2innergn} and \eqref{eq: f gradf gn} into \eqref{eq: 3/2-n-1/n-2},
  which becomes \eqref{eq: 1/8+1/4}.
\end{proof}

We shall eliminate the indeterminate sign term left in Lemma~\ref{le: 1/8+1/4} without treatment via a temporal integration of \eqref{eq: 1/8+1/4} over $(0,s)\in(0,T_{\max})$, which gives us the desired temporal-spatial $L^2_{\loc}$-estimate of  $\Delta v$.  

\begin{lemma}
  \label{le: deltavL2}
  Let $(u_0,v_0,w_0)\in\mathcal{B}(\tilde{m}, A)$.
  Then there exists a constant $C > 0$ such that
  \begin{equation} 
    \label{eq: 1/8deltavL2}
  \begin{aligned}
    &\quad \frac{1}{8} \int_0^s\|\Delta v\|_{L^2(B_\rho)}^2 \dd t
    + \frac{1}{4} \int_0^s \|\nabla v\|_{L^2(B_\rho)}^2 \dd t\\
    &\leq 49\rho^2\int_0^s\|\nabla f\|_{L^2(B_\rho)}^2 \dd t 
    +  CB^{4/(n+2)}\int_0^s \|f\|_{W^{1,2}(\Omega)}^{2n/(n+2)}\dd t\\
    &\quad + \rho\sqrt{m}\int_0^s\|g\|_{L^2(B_\rho)} \dd t
      + CM^{2/(n+4)}\|f\|_{L^2(\Omega)}^{2(n+3)/(n+4)} \\
    &\quad + CM^{2/(n+4)}\|\Delta v - v\|_{L^2(\Omega)}^{2(n+3)/(n+4)} 
     + C(M^2 + B^2)\rho^{n-2\kappa}(1+s)
  \end{aligned} 
  \end{equation}
  holds for all $\rho\in (0,R)$ and $s\in(0,T_{\max})$.
\end{lemma}

\begin{proof}
  Recalling \eqref{eq: wvrr2n-2},
  we compute and estimate by H\"older inequality  
  \begin{align*}
    I &:= \omega_n\int_0^s\int_0^\rho \frac{1}{\eta^{n-1}}\int_0^\eta (w v_r)_tr^\alpha\dd r \dd\eta\dd t \\
    &\;= - \frac{1}{2}\|\nabla v\|_{L^2(B_\rho)}^2 
    + \frac{1}{2}\|v\|_{L^2(B_\rho)}^2 
    - \omega_n\int_0^\rho \frac{n-1}{\eta^{n-1}}\int_0^\eta v^2r^{\alpha-1}\dd r\dd\eta\\
    &\quad\;\,- \omega_n\int_0^\rho \frac{1}{\eta^{n-1}}\int_0^\eta fv_rr^\alpha\dd r\dd\eta 
     - \omega_n\int_0^\rho \frac{1}{\eta^{n-1}}\int_0^\eta (w_0 v_{0r})r^\alpha\dd r \dd\eta\\
    &\;\leq - \frac{1}{2}\|\nabla v\|_{L^2(B_\rho)}^2 
     + \frac{1}{2}\|v\|_{L^2(B_\rho)}^2 
     + \rho\|f\|_{L^2(B_\rho)}\|\nabla v\|_{L^2(B_\rho)}
    + \rho\|w_0\|_{L^2(B_\rho)}\|\nabla v_0\|_{L^2(B_\rho)}
  \end{align*}
  for all $\rho\in(0,R)$ and $s\in(0,T_{\max})$.
  Using \eqref{eq:v2innergn}, \eqref{eq: nablav2<deltav} and \eqref{eq: initial constraints}, 
  we have 
  \begin{equation*}
    \begin{aligned}
      I &\leq \rho CM^{2/(n+4)}\|f\|_{L^2(\Omega)}\|\Delta v - v\|_{L^2(\Omega)}^{(n+2)/(n+4)} 
      + \rho A^2 
      + C(M^2 + B^2)\rho^{n-2\kappa}\\
      &\leq CM^{2/(n+4)}\|f\|_{L^2(\Omega)}^{2(n+3)/(n+4)} 
      + CM^{2/(n+4)}\|\Delta v - v\|_{L^2(\Omega)}^{2(n+3)/(n+4)}\\ 
      &\quad + C(M^2 + B^2)\rho^{n-2\kappa}
    \end{aligned}
  \end{equation*}
  for all $\rho\in(0,R)$ and $s\in(0,T_{\max})$.
  This entails \eqref{eq: 1/8deltavL2} by integrating \eqref{eq: 1/8+1/4} over $(0,s)\subset(0,T_{\max})$.
\end{proof}

\subsection{Refined spatial-temporal estimates of \texorpdfstring{$uv$}{uv}} 
Derivation of a closed form of certain integral inequality for the quantity $\int_\Omega uv$ relies on Lemma~\ref{le: uvL1smallball} and Lemma~\ref{le: deltavL2}, 
both of which state that for $s\in(0,T_{\max})$, the integral $\int_0^s\int_\Omega uv$ has an upper bound in the form a combination of dissipation rate terms and concentration components from energy.

We note by Young inequality 
  \begin{equation} 
    \label{eq: a young inequality}
  a^\gamma b^{1-\gamma} 
  = a^{\gamma}b^{\frac{\gamma(1-\lambda)}{\lambda}}
  \cdot b^{\frac{(\lambda-\gamma)}{\lambda}}
  \leq a^\lambda b^{1-\lambda} + b 
  \quad \text{for all } a, b \geq 0 \text{ and } 0 < \gamma \leq \lambda \leq 1,
  \end{equation}
which will be frequently used. 
For brevity, we also abbreviate $\Omega_T := (0,T)\times\Omega$ for $T>0$.

\begin{lemma}
  \label{le: 1/24mathcalF}
  Let $(u_0,v_0,w_0)\in\mathcal{B}(\tilde{m}, A)$. Then there exist $C > 0$ such that 
  \begin{equation*}
    \begin{aligned}
      \|uv\|_{L^1(\Omega_s)}  
      &\leq C\rho^2\|\nabla f\|_{L^2(\Omega_s)}^2 
      + C\sqrt{ms}\|g\|_{L^2(\Omega_s)} 
       +  CB^{4/(n+2)}s^{2/(n+2)} \|f\|_{L^2((0,s);W^{1,2}(\Omega))}^{2n/(n+2)}\\
      &\quad + CM^{2/(n+4)}\|\Delta v - v\|_{L^2(\Omega)}^{2(n+3)/(n+4)}
      + CM^{2/(n+4)} \|f\|_{L^2(\Omega)}^{2(n+3)/(n+4)}\\
      &\quad + C(M^2+B^2)(1+s)\rho^{n-2\kappa}
    \end{aligned}
   \end{equation*}
   holds for all $\rho\in (0, R)$ and $s\in(0,T_{\max})$.
\end{lemma}

\begin{proof}
  Lemma~\ref{le: uvL1smallball} and Lemma~\ref{le: deltavL2} entail
  there exists $C > 0$ such that 
  \begin{equation} 
    \label{eq: 1/16+3/4}
  \begin{aligned}
    &\quad \frac{1}{16}\int_0^s\int_{\Omega} uv \dd x \dd t 
    + \frac{1}{4} \int_0^s \|\nabla v\|_{L^2(B_\rho)}^2 \dd t
    - \frac{1}{8}\int_0^s\|v\|_{L^2(B_\rho)}^2 \dd t\\
    &\leq 49\rho^2\int_0^s\|\nabla f\|_{L^2(B_\rho)}^2 \dd t 
    +  CB^{4/(n+2)}\int_0^s \|f\|_{W^{1,2}(\Omega)}^{2n/(n+2)}\dd t\\
    &\quad + \rho\sqrt{ms}\int_0^s\|g\|_{L^2(B_\rho)} \dd t
     + C(B^2 + M^2)\rho^{n-2\kappa}(1 + s)\\
    &\quad + CM^{2/(n+4)}\|\Delta v - v\|_{L^2(\Omega)}^{2(n+3)/(n+4)} 
    + CM^{2/(n+4)}\|f\|_{L^2(\Omega)}^{2(n+3)/(n+4)} \\
    &\quad + C M^{4/(n+4)}\|\Delta v - v\|_{L^2(\Omega)}^{2(n+2)/(n+4)} 
    + C M^{4/(n+4)}\|f\|_{L^2(\Omega)}^{2(n+2)/(n+4)}
  \end{aligned} 
\end{equation}
  holds for all $\rho\in (0, R)$ and $s\in(0,T_{\max})$.
  Using \eqref{eq:v2innergn}, \eqref{eq: a young inequality} and H\"older inequality, we reduce \eqref{eq: 1/16+3/4} into 
  \begin{equation}
    \begin{aligned}
      \frac{1}{16}\|uv\|_{L^1(\Omega_s)} 
      &\leq 49\rho^2\|\nabla f\|_{L^2(\Omega_s)}^2
      +  CB^{4/(n+2)}s^{2/(n+2)}\|f\|_{L^2((0,s);W^{1,2}(\Omega))}^{2n/(n+2)}\\
      &\quad + \rho\sqrt{ms}\|g\|_{L^2(\Omega_s)}
       + C(B^2 + M^2)\rho^{n-2\kappa}(1 + s)\\
      &\quad + CM^{2/(n+4)}\|\Delta v - v\|_{L^2(\Omega)}^{2(n+3)/(n+4)} 
      + CM^{2/(n+4)}\|f\|_{L^2(\Omega)}^{2(n+3)/(n+4)}      
    \end{aligned}
  \end{equation}
  holds for all $\rho\in (0, R)$ and $s\in(0,T_{\max})$.
\end{proof}

As the third advantage of the Poho\v{z}aev-type test procedure that generates a small factor $\rho^2$ in the critical term $\rho^2\|\nabla f\|_{L^2(\Omega_s)}^2$,
we shall see that $\rho = \rho(s)$ can be adaptively selected depending on $\|\nabla f\|_{L^2(\Omega_s)}$ 
to effectively convert this quadratic term into a sublinear power of the dissipation rate quantity $\|\nabla f\|_{L^2(\Omega_s)}^2$.

\begin{lemma}
  \label{le: 0suv uvtheta}
  Let $(u_0,v_0,w_0)\in\mathcal{B}(\tilde{m}, A)$. 
  Then there exist $\theta\in(0,1)$ and $C > 0$ such that 
  \begin{equation}
    \label{eq: 0suv uvtheta}
    \int_0^s\int_\Omega uv\dd x\dd t 
    \leq C(M^2+B^2+1)(1+s) \left(\mathcal{F}(0) + \int_\Omega uv - \int_\Omega u\ln u + 1\right)^\theta
  \end{equation}
  holds for all $s\in(0,T_{\max})$, 
  where $\mathcal{F}$ is specified as in \eqref{sym: mathacl F D}.
\end{lemma}

 \begin{proof} 
 Put  
 \begin{displaymath}
  \rho(s) :=
  \min\left\{\frac{R}{2}, \|\nabla f\|_{L^2(\Omega_s)}^{-\frac{1}{2\kappa-n}}\right\} 
  \quad \text{for } s\in(0,T_{\max}),
 \end{displaymath}
 and let  
 \begin{displaymath}
  \theta := 
  \max\left\{ 
    \frac{n+3}{n+4} , \vartheta\right\}\quad\text{with } \vartheta := 1- \frac{1}{2\kappa-n} \in (0,1)\text{ due to } \kappa > n-2.
 \end{displaymath}
 In the case of $\rho(s) < R/2$ for some $s\in(0,T_{\max})$,
 we may infer the existence of $ C > 0$ from Lemma~\ref{le: 1/24mathcalF} such that
 \begin{equation*}
  \begin{aligned}
    \frac{1}{C}\|uv\|_{L^1(\Omega_s)}  
    &\leq \|\nabla f\|_{L^2(\Omega_s)}^{2\vartheta} 
    + \sqrt {ms}\|g\|_{L^2(\Omega_s)} 
    + B^{4/(n+2)}s^{2/(n+2)} \|f\|_{L^2((0,s);W^{1,2}(\Omega))}^{2n/(n+2)}\\ 
    &\quad + M^{2/(n+4)}\|\Delta v - v\|_{L^2(\Omega)}^{2(n+3)/(n+4)}
      + M^{2/(n+4)} \|f\|_{L^2(\Omega)}^{2(n+3)/(n+4)}\\ 
    &\quad + (M^2 + B^2)(1+s)\|\nabla f\|_{L^2(\Omega_s)}
  \end{aligned}
 \end{equation*} 
 holds.
Using the Young inequality \eqref{eq: a young inequality}, we estimate
\begin{equation*}
  \begin{aligned}
    \frac{1}{C}\|uv\|_{L^1(\Omega_s)}  
    &\leq \|\nabla f\|_{L^2(\Omega_s)}^{2\theta} + 1
    + m^{(1-\theta)}s^{1-\theta}\|g\|_{L^2(\Omega_s)}^{2\theta} + ms \\
    &\quad + B^{2(1-\theta)}s^{1-\theta} \|f\|_{L^2((0,s);W^{1,2}(\Omega))}^{2\theta} 
    + B^2s\\ 
    &\quad + M^{2(1-\theta)}\|\Delta v - v\|_{L^2(\Omega)}^{2\theta} + M^2
      + M^{2(1-\theta)} \|f\|_{L^2(\Omega)}^{2\theta} + M^2\\ 
    &\quad + (M^2 + B^2)(1+s)\|\nabla f\|_{L^2(\Omega_s)}^{2\theta} + (M^2 + B^2)(1+s)\\
    &\leq 100(M^2+B^2+1)(1+s)\left(\|f\|_{L^2((0,s);W^{1,2}(\Omega))}^{2\theta} + \|g\|_{L^2(\Omega_s)}^{2\theta}\right)\\
    &\quad + 2(M^2+B^2+1)(1+s)\left(\|\Delta v - v\|_{L^2(\Omega)}^{2\theta} + \|f\|_{L^2(\Omega)}^{2\theta}\right)\\ 
    &\quad + 10(M^2+B^2+1)(1+s).
  \end{aligned}
 \end{equation*} 
 Write $\Pi = \Pi(s) := 1000(M^2+B^2+1)(1+s)$. 
 Using the concavity of the mapping $[0,\infty)\ni\tau\mapsto\tau^\theta\in[0,\infty)$,
 we have 
 \begin{equation*}
  \begin{aligned}
    \frac{1}{C}\|uv\|_{L^1(\Omega_s)}  
    &\leq \Pi\left(\|f\|_{L^2((0,s);W^{1,2}(\Omega))}^{2} 
    + \|g\|_{L^2(\Omega_s)}^{2} 
    + \frac{\|\Delta v - v\|_{L^2(\Omega)}^{2}}{2} 
    + \frac{\|f\|_{L^2(\Omega)}^{2}}{2} + 1\right)^\theta\\
    &\leq \Pi\left(\int_0^s\mathcal{D}(t)\dd t 
    + \frac{\|\Delta v - v\|_{L^2(\Omega)}^{2}}{2} 
    + \frac{\|f\|_{L^2(\Omega)}^{2}}{2} + 1\right)^\theta\\
    & = \Pi \left(\mathcal{F}(0) - \mathcal{F}(s) 
    + \frac{\|\Delta v - v\|_{L^2(\Omega)}^{2}}{2} 
    + \frac{\|f\|_{L^2(\Omega)}^{2}}{2} + 1\right)^\theta\\  
    & = \Pi\left(\mathcal{F}(0) + \int_\Omega uv - \int_\Omega u\ln u + 1\right)^\theta,
  \end{aligned}
 \end{equation*}
which warrants \eqref{eq: 0suv uvtheta} in the case of $\rho(s) < R/2$ for some $s\in(0,T_{\max})$.

On the other hand, if $\rho(s) = R/2$ for some $s\in(0,T_{\max})$, 
then Lemma~\ref{le: 1/24mathcalF} ensures the existence of $ C > 0$ with the property that
 \begin{equation*}
  \begin{aligned}
    \frac{1}{C}\|uv\|_{L^1(\Omega_s)}  
    &\leq \|\nabla f\|_{L^2(\Omega_s)}^{2\vartheta} 
    + \sqrt {ms}\|g\|_{L^2(\Omega_s)} 
    + B^{4/(n+2)}s^{2/(n+2)} \|f\|_{L^2((0,s);W^{1,2}(\Omega))}^{2n/(n+2)}\\ 
    &\quad + M^{2/(n+4)}\|\Delta v - v\|_{L^2(\Omega)}^{2(n+3)/(n+4)}
      + M^{2/(n+4)} \|f\|_{L^2(\Omega)}^{2(n+3)/(n+4)}\\ 
    &\quad + 2^{2\kappa-n}R^{n-2\kappa}(M^2 + B^2)(1+s)
  \end{aligned}
 \end{equation*} 
 holds. Therefore, \eqref{eq: 0suv uvtheta} can be verified in the same manner as above.
\end{proof}

Now we are in a position to show Proposition~\ref{prop: an integral inequality}.
\begin{proof}[Proof of Proposition~\ref{prop: an integral inequality}]
  It is clear that both $M$ and $B$ can be bounded from above by $\tilde{m} + A$ up to a positive multiple, according to the definitions of $M$ and $B$ at \eqref{eq: constraint of mass} and \eqref{eq: pointwise estimates}, respectively. 
  So Lemma~\ref{le: 0suv uvtheta} entails Proposition~\ref{prop: an integral inequality}.
\end{proof}

\section{Finite-time blowup enforced by low initial energy}
\label{sec: finite-time blowup}

This section is devoted to the proof of Theorem~\ref{thm: low energy enforced finite-time blowup}. 
For initial data $(u_0, v_0,w_0)$ with low energy $\mathcal{F}(u_0,v_0,w_0)$, we may deduce a superlinear differential inequality involving the quantity $\int_0^s\int_\Omega uv$ from Proposition~\ref{prop: an integral inequality},
which implies that $(u,v,w)$ cannot exist globally.

\begin{proof}[Proof of Theorem~\ref{thm: low energy enforced finite-time blowup}]
  Fix $\tilde{m} > 0$ and $A > 0$.  
  Proposition~\ref{prop: an integral inequality} provides us constants $C>0$ and $\theta\in(0,1)$ such that 
  \begin{equation}
    \label{eq: intuvf0theta}
    \int_0^s\int_\Omega uv\dd x\dd s 
    \leq C(\tilde{m} + A + 1)^2(1+s) \left(\mathcal{F}(0) + \int_\Omega uv - \int_\Omega u\ln u + 1\right)^\theta
  \end{equation} 
  holds for all $s\in(0,T_{\max})$.
  Fix 
  \begin{equation}
    \label{sym: ell}
    \ell > 2C^{\frac{1}{1-\theta}}(\tilde{m}+A+1)^{\frac{2}{1-\theta}} + 1.
  \end{equation}
  We claim that if 
  \begin{equation}
    \label{eq: conditions on initial energy}
    \mathcal{F}(0) < -\ell - |\Omega|/e := - K(\tilde{m},A),
  \end{equation}
  then 
  \begin{equation}
    \label{eq: Tmax<TmA}
    \begin{aligned} 
    T_{\max} 
    &\leq \left(2^{\frac{\theta-1}{\theta}} - C^{\frac{1}{\theta}}(\tilde{m}+A+1)^{\frac{2}{\theta}}\ell^{\frac{\theta-1}{\theta}}\right)^{\frac{\theta}{\theta-1}} - 1
    =: T(\tilde{m}, A).
    \end{aligned}
  \end{equation}
  We note that the choice of $\ell$ according to \eqref{sym: ell} warrants $T(\tilde{m}, A) \in(1,\infty)$.

  Suppose (by absurdum) that \eqref{eq: Tmax<TmA} is violated.
  Then, we have $T_{\max} > T(\tilde{m},A)$ and in particular, 
  the function
  \begin{equation*}
    \Psi (s) := \int_0^s\int_\Omega u(v-\ln u)\dd x\dd t + \mathcal{F}(0)s + \ell s
    \quad \text{for } s\in(0,T_{\max})
  \end{equation*}
  is well-defined with 
  \begin{equation}
    \label{eq: psiTmA is finite}
    \Psi(T(\tilde{m},A)) < \infty,
  \end{equation}
  according to Proposition~\ref{prop: local existence and uniqueness}.
Clearly, $\Psi\in C^0([0,T_{\max}))\cap C^1((0,T_{\max}))$ is strictly increasing, 
since the dissipation of energy \eqref{eq: energyequation} implies 
\begin{equation}
  \label{eq: psi'>0}
  \begin{aligned}
  \Psi'(s) 
  &= \int_\Omega uv - \int_\Omega u\ln u + \mathcal{F}(0) + \ell\\
  &\geq \mathcal{F}(0) - \mathcal{F}(s) + \ell \geq \ell
  \quad \text{for } s\in(0,T_{\max})
  \end{aligned}
\end{equation}
and particularly
\begin{equation}
  \label{eq: psi > ls}
  \Psi(s) \geq \ell s
  \quad \text{for } s\in(0,T_{\max}).
\end{equation}
Using \eqref{eq: intuvf0theta} and \eqref{eq: conditions on initial energy}, we arrive at 
\begin{equation}
  \label{eq: a differential inequality}
  \begin{aligned} 
  \Psi 
  &\leq C(\tilde{m}+A+1)^2(1+s)\left(\mathcal{F}(0) + \int_\Omega uv - \int_\Omega u\ln u + 1\right)^\theta\\
  &\quad + (\mathcal{F}(0) + |\Omega|/e + \ell)s \\
  &\leq C(\tilde{m}+A+1)^2(1+s)\Psi'^\theta(s)
  \quad \text{for } s\in(0,T_{\max}),
  \end{aligned}
\end{equation}
where we also use the elementary inequality $\xi\ln \xi \geq -1/e$ valid for all $\xi > 0$.
The positivity \eqref{eq: psi'>0} of $\Psi'$ and the pointwise lower estimates \eqref{eq: psi > ls} allows us to invert the inequality \eqref{eq: a differential inequality}, 
and conclude that the mapping $[1,T(\tilde{m}, A))\ni s\mapsto\Psi(s)\in[\ell,\infty)$ is a supersolution of 
the nonautonomous ordinary differential equation with superlinear growth 
\begin{equation}
  \label{sys: NODE}
  \begin{cases}
  \Phi' = C^{-1/\theta}(\tilde{m}+A+1)^{-2/\theta}(1+s)^{-1/\theta}\Phi^{1/\theta}, & s>1,\\
  \Phi(1) = \ell.
  \end{cases}
\end{equation} 
The function 
\begin{equation*}
  \begin{aligned}
  \Phi(s) = \Big(\ell^{\frac{\theta-1}{\theta}} 
  & - 2^{\frac{\theta-1}{\theta}}C^{-\frac{1}{\theta}}(\tilde{m}+A+1)^{-\frac{2}{\theta}} \\
  & + (1+s)^{\frac{\theta-1}{\theta}}C^{-\frac{1}{\theta}}(\tilde{m}+A+1)^{-\frac{2}{\theta}} 
  \Big)^{\frac{\theta}{\theta-1}}
  \quad\text{for } s\in[1,T(\tilde{m},A)),
  \end{aligned}
\end{equation*} 
solves \eqref{sys: NODE}
with the property that 
\begin{equation}
  \label{eq: phi goes to infinite}
  \Phi(s) \nearrow \infty, \quad\text{as } s\nearrow T(\tilde{m},A).
\end{equation}
While a direct comparison argument ensures 
\begin{equation*}
  \Psi(s) \geq \Phi(s) 
  \quad \text{for } s\in(1,T(\tilde{m}, A)),
\end{equation*}
which is absurd in view of \eqref{eq: psiTmA is finite} and \eqref{eq: phi goes to infinite}.
\end{proof}

\section{Initial data with large negative energy}
\label{sec: initial data}

This section is devoted to the existence of initial data with low energy.

\begin{proof}[Proof of Theorem~\ref{thm: dense initial data for blowup}]
  Let $\phi\in C_0^\infty(\mathbb R^n)$ be a radially symmetric and nonnegative function, 
which is compactly supported in $B_1$ such that $\int_{\mathbb R^n}\phi = 1$.
Fixing $\gamma>0$ and $\eta_\star = \min\{1/2,R\}$,
we define for $\eta\in(0, 1)$,
  \begin{displaymath}
    u_\eta(x) := u_0(x) 
    + \left(\ln\frac1\eta\right)^{2\gamma}\eta^{-\frac{n}{2}-2}\phi\left(\frac{x}{\eta}\right),
    \quad x\in\overline{\Omega},
  \end{displaymath}
  and 
  \begin{displaymath}
    v_\eta(x) := v_0(x) 
    + \left(\ln\frac1\eta\right)^{-\gamma} 
    \eta^{2-\frac{n}{2}}
    \phi\left(\frac{x}{\eta}\right),\quad x\in\overline{\Omega},
  \end{displaymath}
  as well as 
  \begin{displaymath}
    w_\eta(x) = w_0(x) + \left(\ln\frac1\eta\right)^{-\gamma} 
    \eta^{2-\frac{n}{2}}\phi\left(\frac{x}{\eta}\right),
    \quad x\in\overline{\Omega}.
  \end{displaymath}
  It is clear that $(u_\eta, v_\eta, w_\eta)$ complies with \eqref{h: initial data} for all $\eta\in(0, \eta_\star)$.
Since 
\begin{displaymath}
  \|v_\eta-v_0\|_{W^{2,2}(\Omega)}
  = \|w_\eta-w_0\|_{W^{2,2}(\Omega)}
  \leq \left(\ln\frac1\eta\right)^{-\gamma} \|\phi\|_{W^{2,2}(B_1)}
\end{displaymath}
are valid for all $\eta\in(0, \eta_\star)$,
$\gamma > 0$ ensures \eqref{eq: vetatov0}.
A direct computation can verify
\begin{displaymath}
  \|u_\eta-u_0\|_{L^p(\Omega)} 
  \leq \left(\ln\frac1\eta\right)^{2\gamma} \eta^{-2-\frac{n}{2}+\frac{n}{p}}\|\phi\|_{L^p(B_1)}
  \quad \text{for all } \eta\in(0, \eta_\star)
\end{displaymath}
and \eqref{eq: uetatou0} follows due to $-2-n/2+n/p>0$ and $n\geq5$.
To verify \eqref{eq: mathcalFuetaveta},
it remains to evaluate $ - \int_\Omega u_\eta v_\eta$ from $\mathcal{F}$ as follows,
  \begin{equation}
  \label{eq: uetaveta}
    \int_\Omega u_\eta v_\eta 
    \geq \left(\ln\frac1\eta\right)^{\gamma}\int_\Omega \phi^{2}\left(\frac{x}{\eta}\right)\eta^{-n}\dd x 
    = \left(\ln\frac1\eta\right)^{\gamma} \|\phi\|_{L^2(B_1)}^2
    \quad \text{for all } \eta\in(0,\eta_\star),
  \end{equation}
  since \eqref{eq: uetatou0} and \eqref{eq: vetatov0} ensure other terms 
  \begin{displaymath}
    \int_\Omega u_\eta \ln u_\eta 
    + \frac{1}{2}\int_\Omega |\Delta v_\eta - v_\eta + w_\eta|^2
    + \frac{1}{2}\int_\Omega |\Delta v_\eta - v_\eta|^2  
  \end{displaymath}
  from $\mathcal{F}$ are uniformly bounded with respect to $\eta\in(0,\eta_\star)$.
In view of $\gamma > 0$, 
we have $\int_\Omega u_\eta v_\eta\to\infty$ as $\eta\searrow0$ according to \eqref{eq: uetaveta} and finish the proof.
\end{proof}

\section*{Acknowledgments}

Supported in part by National Natural Science Foundation of China (No. 12271092, No. 11671079) and the Jiangsu Provincial Scientific Research Center of Applied Mathematics (No. BK20233002).

\end{document}